\documentclass[a4paper,11pt]{article}
\usepackage{latexsym}
\usepackage{amssymb}
\usepackage{bbm}
\usepackage{enumitem}
\usepackage{amsfonts}
\usepackage{amsmath}
\usepackage{amsthm}
\usepackage{cases}
\usepackage[pdftex]{graphicx} %otherwise
\usepackage[paper=a4paper,left=30mm,right=20mm,top=25mm,bottom=30mm]{geometry}

\newenvironment{@abssec}[1]{%
	\if@twocolumn
	
	\section*{#1}%
	\else
	
	\vspace{.05in}\footnotesize
	\parindent .2in
	{\upshape\bfseries #1. }\ignorespaces
	\fi}

{\if@twocolumn\else\par\vspace{.1in}\fi}

\newenvironment{keywords}{\begin{@abssec}{\keywordsname}}{\end{@abssec}}

\newenvironment{AMS}{\begin{@abssec}{\AMSname}}{\end{@abssec}}

\newcommand\keywordsname{Key words}
\newcommand\AMSname{AMS subject classifications}
\newcommand\AMname{AMS subject classification}
\newtheorem{theorem}{Theorem}

\newtheorem{proposition}[theorem]{Proposition}

\def\qed{\vbox{\hrule height0.6pt\hbox{%
			\vrule height1.3ex width0.6pt\hskip0.8ex
			\vrule width0.6pt}\hrule height0.6pt
}}

\title{The principal eigenfunction of the Dirichlet Laplacian with prescribed numbers of critical points on the upper half of a topological torus
\footnotetext[3]{This research was partially supported by the Grants-in-Aid
for Scientific Research (B) ($\sharp$ 18H01126 and $\sharp$ 17H02847) of
Japan Society for the Promotion of Science.}\footnotemark[3]}

\author{Putri Zahra Kamalia\footnote{Corresponding author.} \footnote{Research Center for Pure and Applied Mathematics, Graduate School of Information Sciences, Tohoku
University, Sendai, 980-8579, Japan ({\tt putrizahrakamalia@gmail.com}${}^*$, {\tt  sigersak@gmail.com}).}\  \  and \  Shigeru Sakaguchi\footnotemark[2]} 

\date{}
\begin{document}

\maketitle

\begin{abstract}
We consider the principal eigenvalue problem for the Laplace-Beltrami operator on the upper half of a topological torus under the Dirichlet boundary
condition. We present a construction of the upper half of a topological torus that admits the principal eigenfunction having  exact numbers of critical points. Furthermore, we manage to identify the locations of all the critical points of the principal eigenfunction explicitly.
\end{abstract}



\section{Introduction}
Let $M^+=(M^+,g)$ be the upper half of a topological torus equipped with a Riemannian metric $g$ written in local coordinates $x = (x^1,x^2)$.  
On $M^+$, we consider the principal eigenvalue problem for the Laplace-Beltrami operator under the Dirichlet boundary condition
\begin{equation}\label{Dirichlet Laplacian}
\begin{cases}
\Delta_g u+\lambda_1u=0\ \mbox{ in }M^+,\\[7pt]
u>0\ \mbox{ in }M^+, \quad u=0\ \mbox{ on }\partial M^+ ,\\[7pt]
 \|u\|_{L^2(M^+)}=1,
\end{cases}
\end{equation}
where $\lambda_1$ is the principal eigenvalue of the Laplace-Beltrami operator $\Delta_g$  on $M^+$ given by
\begin{equation}\label{delta u}
\Delta_{g} u = \textrm{div} (\nabla_{g} u)= \sum_{i=1}^2 \frac{1}{\sqrt{|g|}} \frac{\partial}{\partial x^i}  \left(\sqrt{|g|} (\nabla_{g} u)^i \right).
\end{equation}
The main objective of this paper is to show the existence of the principal eigenfunction of \eqref{Dirichlet Laplacian}  on the upper half of a topological torus with  exact numbers of critical points. Furthermore, we aim to disclose the locations of all the critical points explicitly.

In \cite{volkmer}, Volkmer considers the eigenvalue problem for the Laplace-Beltrami operator on the standard torus  $T^2$ embedded in $\mathbb{R}^3$. In \cite[Theorem 1, p. 825]{volkmer}, he constructs the system of the eigenfunctions together with the eigenvalues on $T^2$. For the existence of the critical points of eigenfunctions on topological tori, we may refer to the work that have been done by \cite{bls, jakobson}. Both studies consider the Liouville metric on the flat torus.
%In \cite[Theorem 1, p. 177]{jakobson}, Jakobson et al. construct a sequence of eigenfunctions of Laplace-Beltrami on the two-dimensional standard tori $T^2$ in such a way that when the corresponding eigenvalues tend to infinity the number of critical points of eigenfunctions remains bounded. On the other hand, in \cite[Theorem 1, p. 10101]{bls} Buhovsky et al. show the opposite result. They construct  a sequence of eigenfunctions of Laplace-Beltrami on the two-dimensional standard tori $T^2$ in such a way that when  the corresponding eigenvalues tend to infinity each constructed eigenfunction has infinite number of isolated critical points. For more general domain, 
In \cite[Theorem 1.2, p. 198]{ep}, Enciso et al. show that for any compact  $d$-dimensional Riemannian manifold, $d\geq3$,  one can find a metric in such a way that the first nontrivial eigenfunction can have as many non-degenerate critical points as one wants. See \cite{jnt, magnanini} for two surveys on geometric properties and critical points of eigenfunctions of the Laplace-Beltrami operator.
 
 In our previous work \cite{ks}, we construct topological tori together with the nonconstant stable stationary solutions of reaction-diffusion problems  with exactly $4n$ critical points whose locations are explicit. The topological tori are constructed by employing the regular perturbations of the standard torus $T^2$. We replace the radius of the tube by a positive periodic function. We refer to the new topological torus as the perturbed torus $T^2_\epsilon$ with a small parameter $\epsilon$. 
 
   In  this paper, we start with considering  \eqref{Dirichlet Laplacian} on the upper half $T^+$ of a standard torus $T^2$. By the symmetry of $T^+$, the principal eigenfunction satisfies an ordinary differential equation, and hence its critical points make one circle on $T^+$. Next, we cut away the lower half of the perturbed torus $T^2_\epsilon$  in \cite{ks}. We call the remaining  surface  by the upper half $T^+_\epsilon$ of the perturbed torus $T^2_\epsilon$.  With the aid of the implicit function theorem, we prove that the principal eigenfunction of \eqref{Dirichlet Laplacian} exists on $T^+_\epsilon$. The construction of the upper half $T^+_\epsilon$ of the perturbed torus $T^2_\epsilon$ provides us an explicit formula for  \eqref{Dirichlet Laplacian}. Hence, it allows us to identify the locations of all the critical points of the principal eigenfunction of \eqref{Dirichlet Laplacian} on $T^+_\epsilon$.
   
   In \cite[Proof of Theorem 2.2, p.6]{ks} we see that  the critical points of stable stationary solutions of reaction-diffusion problems on $T^2$ consist of the two circles $\partial T^+=T^2\cap \lbrace x_3=0\rbrace$. When we perform the regular perturbation $T^2_\epsilon$ of $T^2$, we obtain that all the critical points of the solutions on $T^2_\epsilon$ lie in  $\partial T^+_\epsilon =T^2_\epsilon\cap \lbrace x_3=0\rbrace $ regardless of the value of $\epsilon$ (see \cite[Proof of Theorem 1.1, p.14]{ks}). 
   %In this paper,  the critical points of the principal eigenfunction \eqref{Dirichlet Laplacian} on $T^+$ are located in the interior of $T^+$. When we perform the same perturbation to $T^+$ as in \cite{ks}, we obtain that all  
   On the other hand, the locations of the critical points of the principal eigenfunctions of \eqref{Dirichlet Laplacian} lie in the interior of $T^+_\epsilon$ because of the Dirichlet boundary condition and they depend on the value of $\epsilon$.  Thus, we need another new argument to examine the locations of the critical points (see Proposition \ref{value c} in Section 4).
%   We adapt the method used in \cite[Proof of Theorem 1.1, p. 11]{ks} to express the derivative of principal eigenfunction of \eqref{Dirichlet Laplacian} on $T^+_\epsilon$. As a result, the derivative of stable stationary solutions in \cite{ks} and the derivative of principal eigenfunction of \eqref{Dirichlet Laplacian} on $T^+_\epsilon$ are represented as the products of a function of one-variable $\varphi$ and a cosine function of variable $\theta$. Those functions of one-variable $\varphi$ satisfy some ordinary differential equations. In   \cite[Proof of Theorem 1.1, p. 11]{ks},  to get zero derivative we  evaluate the value of the definite integral  of that ordinary differential equation on a certain interval.  On the other hand, in this paper we directly evaluate the value of the ordinary differential equation at a certain point.  We summarize our results in the following theorem.

The main result of this paper states the following.
   \begin{theorem}\label{main theorem}
   	There exists a number $N \in \mathbb{N}$ such that, for each $n \geq N$, a perturbation $M^+$ of the upper half of a standard torus $T^2$ is constructed in such a way that the principal eigenfunction $u$ of \eqref{Dirichlet Laplacian} has exactly $2n$ critical points.
   	\end{theorem}
   
The paper is organized as follows. In Section 2, we consider problem \eqref{Dirichlet Laplacian}  when  $M^+$   is replaced by the upper half $T^+$ of a standard torus $T^2$. We show that the set of the critical points of the principal  eigenfunction of \eqref{Dirichlet Laplacian} equals  a circle in $T^+$. In Section 3, we construct a perturbation $T^+_\epsilon$ of  $T^+$ together with the principal eigenfunction on $T^+_\epsilon$. %whose properties are similar to the principal eigenfunction on $T^+$. 
In Section 4, we prove Theorem \ref{main theorem} along with the locations of all the critical points of the principal eigenfunction on $T^+_\epsilon$.

\setcounter{equation}{0}
\setcounter{theorem}{0}
 
   \section{The upper half of a standard torus}
   The parameterization of the upper  half $T^+$ of a standard torus $T^2$ is given by
   \begin{equation}\label{upper half of standard tori}
   \begin{cases}
    x_1=(R+r \cos\varphi ) \cos \theta , \\
    x_2=(R+r \cos\varphi ) \sin \theta,  \hspace{1cm} ( ( \varphi,\theta) \in I^+:=[0,\pi]\times S^1)\\ 
    x_3=r \sin \varphi,
   \end{cases}
   \end{equation}
   where $R, r$ are constants with $R>r>0$ and $S^1$ is the unit circle.
   
   Set $x^1 = \varphi, \, x^2 = \theta$.  Then, the Riemannian metric $g=(g_{ij})$ is given by 
   \begin{equation}
  (g_{ij})_{i,j=1,2} =
   \begin{pmatrix}
   r^2 & 0 \\[0.5cm]
   0 &  (R+r\cos \varphi)^2
   \end{pmatrix}.
   \end{equation}
   The Riemannian gradient $\nabla_{g}u$ of $u$ with respect to $g$ on $T^+$ is given by
    \begin{equation}
   \nabla_{g} u =
   \begin{pmatrix}
   \displaystyle \dfrac{1}{r^2} \partial_\varphi u \\[0.5cm]
   \dfrac{1}{(R+r\cos \varphi)^2}  \partial_\theta u
   \end{pmatrix}.
   \end{equation}
    The area element $d\sigma$ of $T^+$ is given by
   \begin{equation}
   %   \begin{cases}
   %   ds^2=\sum_{i,j=1}^{2} g_{ij}dx^idx^j=r^2 d\varphi^2+ (R+r\cos \varphi)^2 d\theta^2,\\
   d\sigma=\sqrt{|g|} d\varphi d\theta=r(R+r\cos \varphi)   d\varphi d\theta.
   %   \end{cases}
   \end{equation}
   Thus, we can express the Laplace-Beltrami operator $\Delta_{g}$ on $T^+$ as
   \begin{equation}\label{Laplace operator for torus}
   \Delta_{g} u= \frac{1}{r^2} u_{\varphi\varphi}+\frac{1}{(R+r\cos \varphi)^2} u_{\theta\theta} - \frac{\sin \varphi}{r(R+r \cos \varphi)} u_{\varphi}.
   \end{equation}
   
   Let $U$ be the principal eigenfunction of \eqref{Dirichlet Laplacian} where $M^+$ is replaced by $T^+$. By the symmetry of $T^+$ with respect to the center of the hole of $T^+$, $U$ is a function of one variable $\varphi$.
  \begin{proposition}\label{proposition1}
The set of critical points of the principal eigenfunction $U$ of \eqref{Dirichlet Laplacian} on $T^+$  equals a circle in $T^+$ centered at a point belonging to the axis of the symmetry of $T^+$.
  \end{proposition}
\begin{proof}
Since $U>0$ in $(0,\pi)$ and $U$ satisfies that for every $\varphi \in (0,\pi)$ 
\begin{equation}
\Bigl((R+r\cos \varphi) U^{\prime}\Bigr)^\prime=-r^2(R+r \cos \varphi)\lambda_1U < 0,
\end{equation}
we see that $(R + r \cos \varphi)U^\prime(\varphi)$ is strictly decreasing in $(0, \pi)$.	By applying  Hopf's boundary point lemma to \eqref{Dirichlet Laplacian} on $T^+$, we have that$$ U^\prime(0) > 0 > U^\prime(\pi). $$These conditions give rise to the existence of a unique point  $\varphi^* \in (0,\pi)$ such that 
\begin{equation}\label{u prime}
U^\prime(\varphi^*)=0 \mbox{ and } U^{\prime\prime}(\varphi^*)<0.
%\begin{cases}
%U^\prime(\varphi^*)=0,\\
%U^{\prime\prime}(\varphi^*)<0.
%\end{cases}
\end{equation}
Hence, we conclude that the set of critical points of $U$ of \eqref{Dirichlet Laplacian} on $T^+$ corresponds to $$(\varphi, \theta) \in \lbrace \varphi^*\rbrace \times S^1.$$
This completes the proof.
\end{proof}
\setcounter{equation}{0}
\setcounter{theorem}{0}
\section{The upper half of a perturbed torus}
%Denote by  $T^+_\epsilon$ the pe
%
%In this section, we slightly perturb the upper half $T^+ $ of standard torus. We denote the perturbation of $T^+$ by $T^+_\epsilon$. 
Let  us introduce  a small perturbation $T^+_\epsilon$ of $T^+$ parameterized by 
 \begin{equation} \label{parametrization of perturbed tori}
\begin{cases}
\ \ x_1=(R+r_\epsilon(\theta) \cos\varphi ) \cos \theta , \\
\ \ x_2=(R+r_\epsilon(\theta) \cos\varphi ) \sin \theta,  \hspace{1cm} ( ( \varphi,\theta) \in I^+)\\ 
\ \ x_3=r_\epsilon(\theta) \sin \varphi,
\end{cases}
\end{equation}
where  $n\in \mathbb{N}$, $r_\epsilon(\theta)=r+\epsilon \sin (n\theta)$,  and the constants $R, r, \varepsilon$ satisfy $R> r+|\epsilon|=\max\limits_{\theta \in S^1}  r_\epsilon(\theta)$. Notice that $T^+_0=T^+$. 

Set $x^1=\varphi,\ x^2=\theta$. Then, the Riemannian metric $g^\epsilon=(g^\epsilon_{ij})$ is given by
\begin{equation}
(g^\epsilon_{ij})_{i,j=1,2} =
\begin{pmatrix}
r_\epsilon^2(\theta) & 0\\[0.5cm]
0 &  (R+r_\epsilon(\theta)\cos \varphi)^2+(r_\epsilon^\prime(\theta))^2 
\end{pmatrix}.
\end{equation}
%\begin{equation*}
%\begin{cases}
% ds_\epsilon^2=\sum_{i,j=1}^{2} g^\epsilon_{ij}dx^idx^j=r_\epsilon^2(\theta) d\varphi^2+ \left[ (R+r_\epsilon(\theta)\cos \varphi)^2+(r_\epsilon^\prime(\theta))^2 \right] d\theta^2,\\
% d\sigma_\epsilon=\sqrt{|g^\epsilon|} d\varphi d\theta=r_\epsilon(\theta) \sqrt{ (R+r_\epsilon(\theta)\cos \varphi)^2+(r_\epsilon^\prime(\theta))^2 }   d\varphi d\theta.
%\end{cases} 
%\end{equation*}
The Riemannian gradient $\Delta_{g^\epsilon}u$ of $u$ with respect to $g^\epsilon$ on $T^+_\epsilon$ 
%in local coordinates $(\varphi,\theta) \in S^1 \times S^1$ 
is given by
\begin{equation}\label{gradient}
\nabla_{g^\epsilon} u =
\begin{pmatrix}
\displaystyle \dfrac{1}{r_\epsilon^2 (\theta)} \partial_\varphi u \\[0.5cm]
\dfrac{1}{(R+r_\epsilon(\theta)\cos \varphi)^2+(r_\epsilon^\prime(\theta))^2}  \partial_\theta u
\end{pmatrix},
\end{equation}
 The area element $d\sigma_\epsilon$ of $T^+_\epsilon$ is given by
 \begin{equation}
 d\sigma_\epsilon=\sqrt{|g^\epsilon|} d\varphi d\theta=r_\epsilon(\theta) \sqrt{ (R+r_\epsilon(\theta)\cos \varphi)^2+(r_\epsilon^\prime(\theta))^2 }   d\varphi d\theta.
 \end{equation}
Hence, the Laplace-Beltrami operator $\Delta_{g^\epsilon}$ on $T^+_\epsilon$ can be expressed as
\begin{equation}\label{LB}
\Delta_{g^\epsilon} u=  \frac 1{r_\epsilon^2(\theta)}u_{\varphi\varphi} +\frac 1{\Phi^2}u_{\theta\theta} + \frac{\Phi_\varphi}{r_\epsilon^2(\theta) \Phi}u_\varphi+\frac{r_\epsilon^\prime(\theta)\Phi-r_\epsilon(\theta)\Phi_\theta}{r_\epsilon(\theta)\Phi^3}u_\theta.
\end{equation}
where we set $\Phi = \Phi(\varphi, \theta) = \sqrt{ (R+r_\epsilon(\theta)\cos \varphi)^2+(r_\epsilon^\prime(\theta))^2 }$.

Let $u^\epsilon=u^\epsilon(\varphi,\theta)$ be the principal eigenfunction of \eqref{Dirichlet Laplacian} where $M^+$ is replaced by $T^+_\epsilon$. Then $u^\epsilon$ satisfies
%\begin{equation}
\begin{numcases}{}
%\dfrac 1{r_\epsilon^2(\theta)}u^\epsilon_{\varphi\varphi} +\dfrac 1{\Phi^2}u^\epsilon_{\theta\theta} + \dfrac{\Phi_\varphi}{r_\epsilon^2(\theta) \Phi}u^\epsilon_\varphi+\dfrac{r_\epsilon^\prime(\theta)\Phi-r_\epsilon(\theta)\Phi_\theta}{r_\epsilon(\theta)\Phi^3}u^\epsilon_\theta+\lambda_1^\epsilon u^\epsilon=0 \mbox{ in } T^+_\epsilon,\label{Dirichlet Laplace epsilon} \\[7pt]
\Delta_{g^\epsilon}u^\epsilon+\lambda_1^\epsilon u^\epsilon=0 \ \mbox{ in } T^+_\epsilon,\label{Dirichlet Laplace epsilon} \\[7pt]
u^\epsilon > 0\ \mbox{ in } T^+_\epsilon, \quad u^\epsilon=0\ \mbox{ on }\partial T^+_\epsilon,\\[7pt]
\|u^\epsilon\|_{L^2(T^+_\epsilon)}=1. \label{l2}
\end{numcases}
%\end{equation}
Notice that $u^0=U$ and $\lambda_1^0=\lambda_1$. 

By applying the implicit function theorem to problem \eqref{Dirichlet Laplace epsilon}--\eqref{l2} as in \cite[Example 3.2, pp. 32--33]{henry}, we see that the principal eigenfunction $u^\epsilon$ on $\overline{T^+_\epsilon}$ is  close to the principal eigenfunction $U$ in $\overline{T^+}$ in $C^2$-topology for sufficiently small $|\epsilon|$. 
\begin{proposition}\label{implicit function theorem}%Let $0 < \alpha < 1$.  
	Let  $u^\epsilon$ be the principal eigenfunction of \eqref{Dirichlet Laplacian} where $M^+$ is replaced by $T^+_\epsilon$. Then,  there exists $\epsilon_0 > 0$ such that for each $|\epsilon| \in (0,\epsilon_0)$,  $\delta_0 (\epsilon) > 0$ with $\lim\limits_{\epsilon\rightarrow 0} \delta_0(\epsilon)=0$ exists and satisfies
	 \begin{equation}
	 \|u^\epsilon -U \|_{C^{2}(I^+)} < \delta_0(\epsilon), \mbox{ if } |\epsilon| \in (0,\epsilon_0),
	 \end{equation}
	where $U$ is the principal eigenfunction of \eqref{Dirichlet Laplacian} on $T^+$ given by {\rm Proposition \ref{proposition1}}. 
\end{proposition}
%We apply the implicit function theorem to  prove Theorem \ref{implicit function theorem}. The arguments of the proof follow those used in .  \\

%Now, we proceed to the proof of Theorem \ref{main theorem}. The arguments in the following proof follow those in \cite[Proof of Theorem 1.1, p. 11-14]{ks}. We made some modification to adjust to our problems.
\setcounter{equation}{0}
\setcounter{theorem}{0}
 \section{Proof of Theorem \ref{main theorem}.}
 Let  $u^\epsilon$ be the principal eigenfunction of \eqref{Dirichlet Laplacian} where $M^+$ is replaced by $T^+_\epsilon$ for $ |\epsilon| \in (0,\epsilon_0)$ as in Proposition \ref{implicit function theorem}.
 %\begin{proof}[Proof of  Theorem \ref{main theorem}]
Then the uniqueness of the principal eigenfunction $u^\epsilon$,  
%By the implicit function theorem, the principal eigenfunction $u^\epsilon$ is uniquely determined in the neighbourhood of $U$. This fact 
together with the symmetry of $T^+_\epsilon$,  provides us the symmetry of $u^\epsilon$ with respect to $T_\epsilon^+ \cap H$ for the following $n$ planes $H$:
\begin{equation}
H=\lbrace -x_1 \sin \theta_k+x_2 \cos \theta_k =0\rbrace\ \mbox{ with } k=0,1,\dots,n-1,
\end{equation}
where $\theta_k=\dfrac{2k+1}{2n}\pi$, and hence
\begin{equation}\label{symmetry}
\dfrac{\partial u^\epsilon}{\partial \theta}= 0\ \mbox{ for every } (\varphi, \theta) \in [0, \pi] \times \lbrace \theta_k, \ k = 0, 1, ..., 2n -1\rbrace.
\end{equation}
From \eqref{u prime} and Proposition \ref{implicit function theorem}, the following conditions hold:
Let $\delta > 0$ be sufficiently small. There exists $\epsilon_1 \in (0, \epsilon_0)$ such that if $|\epsilon|\leq \epsilon_1$, then
\begin{equation}\label{detivative}
\begin{cases}
\dfrac{\partial u^\epsilon}{\partial \varphi}>0 \ \mbox{ for every } \varphi\in[0,\varphi^*-\delta],\\[12pt]
\dfrac{\partial u^\epsilon}{\partial \varphi}<0 \  \mbox{ for every } \varphi\in[\varphi^*+\delta, \pi],\\[12pt]
\dfrac{\partial^2 u^\epsilon}{\partial \varphi^2}<0 \ \mbox{ for every } \varphi\in[\varphi^*-2\delta,\varphi^*+2\delta].
\end{cases}
\end{equation}
These conditions imply the  existence  of a unique point $\hat{\varphi}(\epsilon, \theta) \in (\varphi^*-\delta,\varphi^*+\delta)$ with $\hat{\varphi}(0, \theta)\equiv \varphi^*$ satisfying, if  $|\epsilon|\leq \epsilon_1$ then
\begin{equation}\label{vhat}
\dfrac{\partial u^\epsilon}{\partial \varphi}(\hat{\varphi}(\epsilon, \theta), \theta)=0\ \mbox{ for every } \theta \in S^1.
\end{equation}

	From \eqref{detivative} and \eqref{vhat}, if $|\epsilon|\leq \epsilon_1$, we reassure that the set of critical points of the function $u^\epsilon$ on $I^+\left(=[0,\pi] \times S^1\right)$ is contained in $(\varphi^*-\delta,\varphi^*+\delta)\times S^1$.
	% in the neighbourhood of $\varphi^*$.  
	To obtain all the critical points of the principal eigenfunction $u^\epsilon$ on $T^+_\epsilon$, it suffices to examine the derivatives of $u^\epsilon$ with respect to $\theta$ up to the second order only for $\varphi \in [\varphi^*-\delta,\varphi^*+\delta]\times S^1$.
	
	We may express  the principal eigenfunction $u^\epsilon$ as
\begin{equation}\label{u epsilon}
u^\epsilon=U+\epsilon\dfrac{\partial u^\epsilon}{\partial \epsilon}\bigg|_{\epsilon=0}+ o(\epsilon) \mbox{ as } \epsilon \rightarrow 0.
\end{equation}
Set $V=\dfrac{\partial u^\epsilon}{\partial \epsilon}\bigg|_{\epsilon=0}$. By the Dirichlet boundary condition of $u^\epsilon$, we have
\begin{equation}\label{boundary V}
V(0,\theta)=V(\pi,\theta)=0 \ \mbox{ for every }\theta \in S^1.
\end{equation}
By \eqref{u epsilon}, the derivatives of $u^\epsilon$ with respect of $\theta$ up to the second order are given by
\begin{equation}\label{uv}
\dfrac{\partial u^\epsilon}{\partial \theta}=\epsilon V_\theta + o(\epsilon) \quad \mbox{ and } \quad \dfrac{\partial^2 u^\epsilon}{\partial \theta^2}=\epsilon V_{\theta \theta}+ o(\epsilon) \quad \mbox{ as } \epsilon \rightarrow 0.
\end{equation}
%To determine all the critical points of $u^\epsilon$ 

Let us differentiate \eqref{Dirichlet Laplace epsilon} with respect to $\epsilon$.
\begin{align*}
%=\frac {\partial}{\partial\epsilon}\left( \Delta_{g^\epsilon} u^\epsilon+\lambda_1^\epsilon u^\epsilon \right)
0=& \frac 1{r_\epsilon^2}u^\epsilon_{\varphi\varphi\epsilon} +\frac {\partial}{\partial\epsilon}\left(\frac 1{r_\epsilon^2}\right) u^\epsilon_{\varphi\varphi} +\frac 1{\Phi^2}u^\epsilon_{\theta\theta\epsilon}+\frac {\partial}{\partial\epsilon}\left( \frac 1{\Phi^2}\right) u^\epsilon_{\theta\theta} \\
&+\frac{\Phi_\varphi}{r_\epsilon^2 \Phi}u^\epsilon_{\varphi\epsilon}+\frac {\partial}{\partial\epsilon}\left(\frac{\Phi_\varphi}{r_\epsilon^2 \Phi}\right) u^\epsilon_\varphi\\
&+\frac{r_\epsilon^\prime\Phi-r_\epsilon\Phi_\theta}{r_\epsilon\Phi^3}u^\epsilon_{\theta\epsilon}+\frac {\partial}{\partial\epsilon}\left(\frac{r_\epsilon^\prime\Phi-r_\epsilon\Phi_\theta}{r_\epsilon\Phi^3}\right) u^\epsilon_\theta+ \frac{\partial \lambda_1^\epsilon}{\partial \epsilon} u^\epsilon+\lambda_1^\epsilon \frac{\partial u^\epsilon}{\partial \epsilon}.
\end{align*}
Then, we set $\epsilon=0$ to obtain
\begin{equation}\label{derivative V}
%\left(F(\epsilon, \rho (\epsilon) )\right)_\epsilon\bigg|_{\epsilon=0} =
\Delta_{g} V + \lambda_1V+ \frac{\partial \lambda_1^\epsilon}{\partial \epsilon}\bigg|_{\epsilon=0} U=\frac 2r \sin (n\theta) \left[-\lambda_1 U+ \dfrac{R\sin \varphi}{2r(R+r\cos \varphi)^2} U_{\varphi} \right]
\end{equation}
for every $(\varphi,\theta) \in I^+$. Multiply both sides of \eqref{derivative V} by $U$ and integrate them on $I^+$ to obtain
\begin{equation}\label{integral V}
\int\limits_{I^+} \left\lbrace \bigg(\Delta_g V+\lambda_1 V\bigg) U+ \dfrac{\partial \lambda_1^\epsilon}{\partial \epsilon}\bigg|_{\epsilon=0} U^2\right\rbrace \,d\sigma=0.
\end{equation}
By using the integration by parts and the fact that $\int\limits_{I^+} U^2 d\sigma=1$, we obtain 
\begin{equation*}
\int\limits_{I^+} \bigg(\Delta_g U+\lambda_1 U\bigg) V d\sigma+ \dfrac{\partial \lambda_1^\epsilon}{\partial \epsilon}\bigg|_{\epsilon=0} \,=0.
\end{equation*}
Since $\Delta_g U+\lambda_1 U=0$, $\dfrac{\partial \lambda_1^\epsilon}{\partial \epsilon}\bigg|_{\epsilon=0}=0$. Hence, for every $ (\varphi,\theta)\in I^+$
\begin{equation} \label{eqv1}
\hspace*{-0.4cm}	\Delta_g V+\lambda_1 V=\dfrac{2}{r}\sin (n\theta)\left[-\lambda_1U +\dfrac{R \sin \varphi}{2r(R+r\cos \varphi)^2}U_\varphi\right].
\end{equation}
The right-hand side of \eqref{eqv1} can be regarded as a product of a function of one variable $\varphi$ and  a function of one variable $\theta$ and all the coefficients of the left-hand side of \eqref{eqv1} are independent of $\theta$. Moreover, the right-hand side of \eqref{eqv1} is infinitely differentiable. Then, by the standard regularity theory for elliptic partial differential equations (see \cite{Gilbarg-Trudinger}), we may differentiate both sides of \eqref{eqv1} with respect to $\theta$ twice to obtain

\begin{equation}\label{persamaanV}
\Delta_g \bigg(\dfrac{V_{\theta \theta}}{n^2}+V\bigg) +\lambda_1  \bigg(\dfrac{V_{\theta \theta}}{n^2}+V\bigg) =0\ \text{ in } T^+.
\end{equation}
In addition, by \eqref{boundary V}
\begin{equation}\label{boundary V ode}
\dfrac{V_{\theta \theta}}{n^2}+V=0\ \text{ on } \partial T^+. 
\end{equation} 
Since the principal eigenfunction $U$ of \eqref{Dirichlet Laplacian} on $T^+$ is unique, \eqref{persamaanV} together with \eqref{boundary V ode} yields that
\begin{equation}\label{V ode}
\dfrac{V_{\theta \theta}}{n^2}+V=c U
\end{equation}
for some constant $c\in \mathbb{R}$. Hence, the function $V$ can be expressed as
\begin{equation}\label{definitionofv}
V(\varphi, \theta)= C_1(\varphi) \cos (n\theta)+C_2(\varphi)\sin (n\theta)+cU\ \text{ for every } (\varphi,\theta)\in I^+,
\end{equation}
for some functions $C_1(\varphi)$, $C_2(\varphi)$ of class $C^2$. By \eqref{boundary V}, we have 
\begin{equation}\label{nilai C}
C_1(0)=C_2(0)=C_1(\pi)=C_2(\pi)=0.
\end{equation}

Let us substitute \eqref{definitionofv} into \eqref{eqv1}. Then, $C_1, C_2$ satisfy the following two ordinary differential equations:
\begin{align}
&C^{\prime \prime}_1(\varphi) -\dfrac{r\sin \varphi}{R+r\cos \varphi}C^\prime_1(\varphi)- B_n(\varphi) C_1(\varphi)=0,\label{odec1}\\
&C^{\prime \prime}_2(\varphi) -\dfrac{r\sin \varphi}{R+r\cos \varphi}C^\prime_2(\varphi)-B_n(\varphi) C_2(\varphi)=A(\varphi),\label{odec2}
\end{align}
where we set
$$
A(\varphi)=2r\left[-\lambda_1 U+ \dfrac{R\sin \varphi}{2r(R+r\cos \varphi)^2} U_{\varphi} \right] \ \mbox{ and }\ B_n(\varphi)=r^2\left\lbrack \dfrac{n^2}{(R+r\cos \varphi)^2} -\lambda_1\right\rbrack.
$$
Choose $N \in \mathbb{N}$ in Theorem \ref{main theorem} as $N>  \sqrt{\lambda_1} (R+r)$. Then, for every $n \geq N$, $B_n(\varphi)$ is positive in $(0,\pi)$. Let $n \geq N$. Hence,  by applying the maximum principle to \eqref{odec1} we have that $C_1 \equiv 0$ .
Therefore, we have 
\begin{equation}\label{definition of v}
V(\varphi, \theta)= C_2(\varphi)\sin (n\theta)+cU\  \text{ for every } (\varphi,\theta)\in I^+.
\end{equation}
%Substituting the boundary conditions \eqref{nilai C} into \eqref{odec2} yields that $C^{\prime\prime}(0)=C^{\prime\prime}(\pi)=0$.

To evaluate $V(\varphi,\theta)$ for  $\varphi\in [\varphi^*-\delta,\varphi^*+\delta]$, we need to evaluate $C_2(\varphi)$ for  $\varphi \in [\varphi^*-\delta,\varphi^*+\delta]$. %By some arguments (see  Proposition \ref{value c} after this proof),
 In Proposition \ref{value c} below, we will show that $C_2(\varphi^*)$ is positive. Then,  
the continuity of $C_2$ gives us $C_2(\varphi)>0$ for every   $\varphi \in [\varphi^*-\delta,\varphi^*+\delta]$, provided that $\delta > 0$ is chosen sufficiently small.

An elementary calculation of  the derivatives of \eqref{definition of v} with respect to $\theta$ up to the second order shows that,  for every $\varphi \in [\varphi^*-\delta,\varphi^*+\delta]$
\begin{equation}
\begin{cases}
V_\theta(\varphi, \theta)\neq 0 \quad \text{ if } \quad  \theta\notin \lbrace \theta_k\big| k=0,1,2,...,2n-1\rbrace,\\
V_{\theta\theta}(\varphi, \theta) \neq 0 \quad \text{if } \quad  \theta\in \lbrace \theta_k\big| k=0,1,2,...,2n-1\rbrace.\\
\end{cases}
\end{equation}
Then, it follows from  \eqref{symmetry} and \eqref{uv} that there exists $\epsilon_2\in (0,\epsilon_1)$ such that if $|\epsilon|<\epsilon_2$,  for every $\varphi\in [\varphi^*-\delta,\varphi^*+\delta]$
\begin{equation*}\label{utt}
\hspace*{-0.0cm}
\begin{cases}
\dfrac{\partial u^\epsilon}{\partial \theta}(\varphi, \theta)\neq 0 \quad \text{ if } \quad  \theta\notin \lbrace \theta_k\big| k=0,1,2,...,2n-1\rbrace, \\[10pt]
\dfrac{\partial^2 u^\epsilon}{\partial \theta^2}(\varphi, \theta)>0\quad \text{if } \quad  \theta \in \mathcal{N}_k, \, k\text{ is an odd number in } \lbrace0,1,..., 2n-1\rbrace,\\[10pt]
\dfrac{\partial^2 u^\epsilon}{\partial \theta^2}(\varphi, \theta)<0 \quad \text{if } \quad \theta \in \mathcal{N}_k , \,k  \text{ is an even number in } \lbrace0,1,..., 2n-1\rbrace,
\end{cases}
\end{equation*}
where  $\mathcal{N}_k =[\theta_k-\delta, \theta_k+\delta]$.

We conclude that the set of critical points of the principal eigenfunction $u^\epsilon$ on $T^+_\epsilon$ corresponds to:
\begin{equation*}
\hspace{-1mm} (\varphi,\theta)\in  \left\lbrace (\hat{\varphi}(\epsilon, \theta_k),\theta_k) \big|\, k = 0, 1, . . . , 2n- 1\right\rbrace,
\end{equation*}
which consists of exactly $2n$ critical points in $T^+_\epsilon$. This completes the proof. \qed

\begin{proposition}\label{value c}
	$C_2(\varphi^*)$ is strictly positive.
\end{proposition}
\begin{proof}
	 Suppose on the contrary that $C_2(\varphi^*)\leq 0$.  We will show that $C^\prime_2(\varphi^*) >0$.
	 
	 Suppose that  $C^\prime_2(\varphi^*) =0$.  Then, from \eqref{odec2} we have
	 \begin{align*}
	 C^{\prime \prime}_2 (\varphi^*)&=B_n(\varphi^*) C_2(\varphi^*)+A(\varphi^*) \\
	 &\leq A(\varphi^*) =-2r\lambda_1 U(\varphi^*)<0.
	 \end{align*}
	 Since $C_2(\pi)=0$, there exists a point $\eta \in (\varphi^*, \pi)$ that satisfies
	 \begin{equation}\label{propc}
	 \begin{cases}
	 C_2 ( \eta) = \displaystyle \min_{[\varphi^*, \pi]} C_2 <0,\\
	 C^\prime_2 ( \eta)=0, \quad C^{\prime \prime}_2 ( \eta)\geq 0.
	 \end{cases}
	 \end{equation}
	 Substituting $\eta$ into the left-hand side of \eqref{odec2} yields
	 \begin{equation}\label{cmiu}
	 C^{\prime \prime}_2(\eta) -\dfrac{r\sin \eta}{R+r\cos \eta}C^\prime_2(\eta)-B_n(\eta) C_2(\eta) > 0.
	 \end{equation}
On the other hand,
	 \begin{equation}\label{Amiu}
	 A(\eta)=2r\left[-\lambda_1 U(\eta)+ \dfrac{R\sin \eta}{2r(R+r\cos \eta)^2} U_{\varphi}(\eta) \right] <0,
	 \end{equation}
	 which is a contradiction. Hence, $C_2^\prime(\varphi^*) \neq 0$.

	 Next, we suppose that $C_2^\prime(\varphi^*) <0$. Since $C_2(\varphi^*) \leq 0$ and $C_2(\pi)=0$, there exists $\eta \in (\varphi^*, \pi)$ satisfying  \eqref{propc}. Exactly in the same way as  in the previous case, we get a contradiction. Eventually, we conclude that  $C_{2}^{\prime} (\varphi^*) >0$. Under this circumstance, we shall prove that $C_2(\varphi)< 0$ for every $\varphi \in  [0, \varphi^*)$, which contradicts $C_2(0)=0$.
	 	 
	 Recall that $A(\varphi^*)<0$. Set
	 \begin{equation}
	 \gamma=\inf \bigg\{\Lambda \in (0,\varphi^*) \, \bigg| \, A(\varphi) < 0\  \mbox{ for every } \varphi \in (\Lambda , \varphi^*)\bigg\}.
	 \end{equation}
	 Then, $0\leq \gamma < \varphi^*$. Note that $\gamma$ is independent of $n$.
	Since $A(0)=0$, it follows that $A(\gamma)=0$. Equation \eqref{odec2} is represented as
	\begin{align}
	\bigg((R+r\cos \varphi)C^{\prime }_2(\varphi)\bigg)^\prime=(R+r\cos \varphi) \bigg[B_n(\varphi) C_2(\varphi)+A(\varphi)\bigg].\label{odec22}
	\end{align}
	
 Since   $C_2(\varphi^*)\leq 0$ and $C_{2}^{\prime} (\varphi^*) >0$, there exists $\beta \in (0,\varphi^*-\gamma)$ such that 
 \begin{equation}
 C_2 (\varphi)<0\  \mbox{ for every }\varphi\in(\varphi^*-\beta,\varphi^*). 
 \end{equation}
 If $\varphi \in (\varphi^*-\beta,\varphi^*)$,  then
	\begin{equation*}
	\int\limits_{\varphi}^{\varphi^*} \bigg((R+r\cos h)C^{\prime }_2(h)\bigg)^\prime dh=\int\limits_{\varphi}^{\varphi^*} (R+r\cos h) \bigg[B_n(h) C_2(h)+A(h)\bigg] dh < 0,
	\end{equation*}
	with the result that
%	$$	(R+r\cos \varphi^*)C^{\prime }_2(\varphi^*)-(R+r\cos \varphi)C^{\prime }_2(\varphi) <0 \text{ for every } \varphi \in (\varphi^*-\delta_n,\varphi^*).$$
%	It indicates that 
$$
(R+r\cos \varphi)C^{\prime }_2(\varphi)> (R+r\cos \varphi^*)C^{\prime }_2(\varphi^*)>0\ \mbox{ for every }\varphi \in (\varphi^*-\beta,\varphi^*).
$$ 
Hence,
	\begin{equation*}
	C^{\prime }_2(\varphi)>0 \mbox{ and } C_2(\varphi) < C_2(\varphi^*)\leq 0\  \mbox{ for every } \varphi \in (\varphi^*-\beta,\varphi^*).
	%\begin{cases}
	%C^{\prime }_2(\varphi)>0  &\mbox{ in } (\varphi^*-\beta,\varphi^*), \\
	%C_2(\varphi) < C_2(\varphi^*)\leq 0 
	%&\mbox{ in } (\varphi^*-\beta,\varphi^*).
	%\end{cases}
	\end{equation*}
	
	Set
	\begin{equation} 
	H=\inf \bigg\{\Lambda \in (\gamma,\varphi^*) \, \bigg| \, C^\prime_2(\varphi) > 0\ \mbox{ for every } \varphi \in (\Lambda , \varphi^*)\bigg\}.
	\end{equation}
	Then $\gamma\le H<\varphi^*$.
		Let us show that $H=\gamma$. %and $C_2^\prime (\varphi)>0$ in $(\gamma, \varphi^*)$. 
	  For this purpose, we suppose that $H>\gamma$.  
	From the definition of $H$, 
	\begin{equation}
	C_2^\prime (\varphi)>0 \mbox{ and } C_2(\varphi) < C_2(\varphi^*)\leq 0\ \mbox{ for every }\varphi\in(H,\varphi^*).
	%\begin{cases}
	%C_2^\prime (\varphi)>0  &\mbox{ in }(H,\varphi^*),\\
	%C_2(\varphi) < C_2(\varphi^*)\leq 0 &\mbox{ in }(H,\varphi^*).
	%\end{cases}
	\end{equation}
	Integrating  \eqref{odec22} in $\varphi$ from $H$ to $\varphi^*$ yields that
		\begin{align*}
	\int\limits_{H}^{\varphi^*} \bigg((R+r\cos h)C^{\prime }_2(h)\bigg)^\prime dh=\int\limits_{H}^{\varphi^*} (R+r\cos h) \bigg[B_n(h) C_2(h)+A(h)\bigg] dh <0, 
	\end{align*}
	and hence
		$$(R+r\cos H)C^{\prime }_2(H)> (R+r\cos \varphi^*)C^{\prime }_2(\varphi^*)>0.$$
	 	Then $C_2^\prime(H)>0$. This contradicts the definition of $H$. Therefore $H=\gamma$.  % It follows that  \begin{equation}
%	 	\begin{cases}
%	 	C_2^\prime (\varphi)>0  &\mbox{ in }(\gamma,\varphi^*),\\
%	 	C_2(\varphi) < C_2(\varphi^*)\leq 0 &\mbox{ in }(\gamma,\varphi^*).
%	 	\end{cases}
%	 	\end{equation}

	 	Since $C_2(\varphi)<0$ for every $\varphi \in (\gamma, \varphi^*)$,  if $\varphi \in (\gamma, \varphi^*)$ then
	 	%Since $C_2(\varphi)<0$ in $(\gamma, \varphi^*)$, for every $\varphi \in (\gamma, \varphi^*)$
	 	\begin{align*}
	 	\int\limits_{\varphi}^{\varphi^*} \bigg((R+r\cos h)C^{\prime }_2(h)\bigg)^\prime dh&=\int\limits_{\varphi}^{\varphi^*} (R+r\cos h) \bigg[B_n(h) C_2(h)+A(h)\bigg] dh \\
	 	&<\int\limits_{\varphi}^{\varphi^*} (R+r\cos h) A(h)\, dh.
	 	\end{align*}
	 	 Let us set
	 	$$K(\varphi)=\displaystyle \int\limits_{\varphi}^{\varphi^*} (R+r\cos h) A(h) \,dh\  \mbox{ for every }\varphi \in (\gamma, \varphi^*).$$ 
	 	Since $A(\varphi)<0$ for every $\varphi \in (\gamma, \varphi^*)$, we have
	 	$$ (R+r\cos \varphi^*)C^{\prime }_2(\varphi^*)-(R+r\cos \varphi)C^{\prime }_2(\varphi)< K(\varphi)<0 \ \mbox{ for every } \varphi \in (\gamma, \varphi^*).$$ 
	 	 Hence, 
	 	$$ (R+r\cos \varphi)C^{\prime }_2(\varphi)>(R+r\cos\varphi^*)C^{\prime }_2(\varphi^*)-K(\varphi)>-K(\varphi)\ \mbox{ for every } \varphi \in (\gamma, \varphi^*).$$ 
	 	Thus,  \begin{equation}\label{c gamma}
	 	C^{\prime }_2(\varphi)> \displaystyle -\dfrac{K(\varphi)}{ R+r\cos \varphi}\quad\mbox{ for every } \varphi \in (\gamma, \varphi^*).
	 	\end{equation}
	 		Therefore, it follows that
	 	\begin{align*}
	 	C_2(\gamma)&=C_2(\varphi^*)-\int\limits_\gamma^{\varphi^*}C^\prime_2(h)\,dh \leq-\int\limits_\gamma^{\varphi^*}C^\prime_2(h)\,dh <  \int\limits_\gamma^{\varphi^*} \dfrac{K(h)}{ R+r\cos h} dh .
	 	\end{align*}
	 	If we set $\tau=-\displaystyle \int\limits_\gamma^{\varphi^*} \dfrac{K(h)}{ R+r\cos h}dh\, (>0)$, $\tau$ is independent of $n$ and $C_2(\gamma)< -\tau$. 
	 	Moreover, letting $\varphi \rightarrow \gamma^+$ in \eqref{c gamma} yields that
	 	\begin{equation}\label{c prime gamma}
	 	C^{\prime }_2(\gamma)\geq \displaystyle -\dfrac{K(\gamma)}{ R+r\cos \gamma}>0. 
	 	\end{equation}
	 	
	 	Next we set
	 	$$H^*=\inf \bigg\{ \Lambda \in (0,\gamma) \, \bigg| \, C^\prime_2 (\varphi) > 0\ \mbox{ for every } \varphi \in (\Lambda,\gamma]\bigg\}.$$
	 	Then $0\leq H^* < \gamma$ and hence
	 	\begin{equation*}
		C^\prime_2 (\varphi) > 0 \mbox{ and } C_2(\varphi)<C_2(\gamma)< -\tau\  \mbox{ for every } \varphi \in  (H^*,\gamma).
	 	%\begin{cases}
	 	%C^\prime_2 (\varphi) > 0 &\mbox{ in } (H^*,\gamma],\\
	 	%C_2(\varphi)<C_2(\gamma)< -\tau &\mbox{ in } (H^*,\gamma).
	 	%\end{cases}
	 	\end{equation*}
	 	Let us show that $H^*=0$. For this purpose, we  suppose  that $H^*>0$. 
	 	Recall that $n\geq N>  \sqrt{\lambda_1} (R+r)$.  Since $ \dfrac{A(\varphi)}{\tau}$ is independent of $n$, we may update $N \in \mathbb{N}$ with a large number such  that if $n \geq N$ we have% $$\displaystyle N^2 > \max_{\varphi \in [0,\pi]} \bigg\{\lambda_1 (R+r)^2, \dfrac{(R+r)^2}{r^2\tau} A(\varphi)+\lambda_1 (R+r)^2 \bigg\},$$ 
	%$$\displaystyle N^2 >  \max_{\varphi \in [0,\pi]}  |A(\varphi) |\dfrac{(R+r)^2}{r^2\tau} +\lambda_1 (R+r)^2.$$ 
%	Then, if $n \geq N$, we have 
 	$$B_n (\varphi) > \dfrac{A(\varphi)}{\tau}\  \mbox{ for every } \varphi \in (0,\pi).$$ 
	Let $n \geq N$. As a consequence,
	\begin{align*}
	\int\limits_{H^*}^{\gamma} \bigg((R+r\cos h)C^{\prime }_2(h)\bigg)^\prime dh&=\int\limits_{H^*}^{\gamma} (R+r\cos h) \bigg[B_n(h) C_2(h)+A(h)\bigg] dh\\
	&<\int\limits_{H^*}^{\gamma} (R+r\cos h) \bigg[-B_n(h) \tau+A(h)\bigg] dh < 0.
	\end{align*}
	Then, we have from \eqref{c prime gamma} that
	$$0< (R+r\cos \gamma)C^{\prime }_2(\gamma)< (R+r\cos H^*)C^{\prime }_2(H^*).$$
	Thus $C^{\prime }_2(H^*)>0$, which contradicts the definition of $H^*$. Therefore $H^*=0$ and hence
	\begin{equation*}
	C^\prime_2 (\varphi) > 0  \mbox{ and } C_2(\varphi)< -\tau\  \mbox{ for every } \varphi \in  (0,\gamma).
	%\begin{cases}
	%C^\prime_2 (\varphi) > 0 &\mbox{ in } (0,\gamma],\\
	%C_2(\varphi)< -\tau &\mbox{ in } (0,\gamma).
	%\end{cases}
	\end{equation*}
	In particular, we have $C_2(0)<0$ which  contradicts the fact that $C_2(0)=0$. Eventually, $C_2(\varphi^*)>0$.
	%\end{proof}
	%	\qed
\end{proof} 
 
\end{document}